\documentclass[11pt,a4paper]{amsart}
\usepackage{amsfonts}
\usepackage{amsthm}
\usepackage{amsmath}
\usepackage{amscd}
\usepackage[latin2]{inputenc}
\usepackage{t1enc}
\usepackage[mathscr]{eucal}
\usepackage{indentfirst}
\usepackage{graphicx}
\usepackage{graphics}
\usepackage{pict2e}
\usepackage{epic}
\usepackage{amssymb}
\usepackage{MnSymbol}
\usepackage[colorlinks,citecolor=red,urlcolor=blue,bookmarks=false,hypertexnames=true]{hyperref} 
\numberwithin{equation}{section}
\usepackage[margin=3cm]{geometry}

\theoremstyle{plain}
\newtheorem{Th}{Theorem}[section]
\newtheorem{Lemma}[Th]{Lemma}
\newtheorem{Cor}[Th]{Corollary}

\newtheorem*{Theorem-non}{The mean-square version of Theorem \ref{Theo1}}
\newtheorem*{Theorem-non2}{Theorem 4.8}
\newtheorem*{Theorem-non3}{Corollary 1.2}
 \theoremstyle{definition}

\newtheorem{Rem}[Th]{Remark}
\newtheorem{?}[Th]{Problem}

\begin{document}

\email{Jkim51@olemiss.edu}
\title{Applications of zero-free regions on averages and shifted convolution sums of Hecke eigenvalues
}
\author{Jiseong Kim}
\begin{abstract} 
By assuming Vinogradov-Korobov type zero-free regions and the generalized Ramanujan-Petersson conjecture, we establish nontrivial upper bounds for almost all short sums of Fourier coefficients of Hecke-Maass cusp forms for $SL(n,\mathbb{Z})$. As applications, we obtain nontrivial upper bounds for the averages of shifted sums involving coefficients of the Hecke-Maass cusp forms for $SL(n,\mathbb{Z})$. Furthermore, we present a conditional result regarding sign changes of these coefficients.
\end{abstract}

\address{University of Mississippi, Department of Mathematics}
\maketitle
\markboth{Jiseong Kim}{\text{Zero-free regions and averages over short intervals}}

\section{Introduction}
\noindent There are many interesting results regarding the behavior of multiplicative functions that concern mean values within short intervals and their applications. For a recent overview of these results, we refer the reader to \cite{MR3} and \cite{GHS}.  The utilization of Hecke relations reveal that the normalized Fourier coefficients $\{\mathcal{A}_{F}(m,1,...,1)\}$ of Hecke-Maass forms for $SL(n,\mathbb{Z})$ (see Subsection 1.2) are multiplicative, and it is conjectured that these are divisor-bounded as indicated by the generalized Ramanujan-Petersson conjecture (see Lemma \ref{Lemma 3.1}). In addition, there have been  extensive research on analytic properties of automorphic L-functions, which is useful for studying the behavior of ${\mathcal{A}_{F}(m,1,...,1)}$. Given this background, one would seek to apply techniques from multiplicative number theory to automorphic L-functions. In this paper, we show that simple methods outlined
 in \cite{Matomaki1} and \cite{MR3} can be applied to study $\{\mathcal{A}_{F}(m,1,...,1)\}$ under the generalized Ramanujan-Petersson conjecture and the $SL(n,\mathbb{Z})$ Vinogradov-Korobov zero-free regions (see subsection \ref{GRP2}, \eqref{Zerofree} respectively). One of the main results in this paper is as follows.
\begin{Th}\label{Theo1} Assume the generalized Ramanujan-Petersson conjecture and the $SL(n,\mathbb{Z})$ Vinogradov-Korobov zero-free regions. Let $X$ be sufficiently large, and let $e^{ (\log X)^{1-\epsilon}} \ll_{\epsilon} h_{1} \leq h_{2} \ll_{\epsilon} X^{1-\epsilon}$ where $\epsilon>0$ is fixed small. Then
\begin{equation}\label{316}\frac{1}{X} \int_{X}^{2X} \left|\frac{1}{h_{1}}\sum_{m=x}^{x+h_{1}} \mathcal{A}_{F}(m,1,...,1)- \frac{1}{h_{2}}\sum_{m=x}^{x+h_{2}} \mathcal{A}_{F}(m,1,...,1)   \right|^{2}dx \ll_{F,\epsilon}  (\log X)^{-2/3+\epsilon}. \end{equation}
\end{Th}

As applications, we obtain nontrivial upper bounds for shifted convolution sums involving $\{\mathcal{A}_{F}(m,1,...,1)\}$ and investigate sign changes of them. Before discussing our results, let us review some relevant studies on holomorphic Hecke cusp forms.

\subsection{Holomorphic Hecke cusp forms}
Let $\mathbb{H}=\{z=x+iy : x \in \mathbb{R}, y \in (0,\infty)\}, G=SL(2,\mathbb{Z}).$ 
Define $j_{\gamma}(z)=(cz+d)$ where 
$\gamma= \left(\begin{matrix}
a & b \\
c & d 
\end{matrix} \right).$
A holomorphic function $f: \mathbb{H} \rightarrow \mathbb{C}$ is called a modular form of weight $k$ if it satisfies the modularity
\begin{equation}f(\gamma z)= j_{\gamma}(z)^{k}f(z)\nonumber\end{equation}
for all $\gamma \in SL(2,\mathbb{Z})$. It is well-known that $f(z)$ has a Fourier expansion at the cusp $\infty$ given by 
\begin{equation}
    f(z)=\sum_{n=0}^{\infty} b_{n}e(nz)
\end{equation}
where $e(z)=e^{2\pi iz},$ and the normalized Fourier coefficients $a(n)$ of $f(z)$ are defined by
\begin{equation}
a(n):=b_{n}n^{-\frac{k-1}{2}}.
\end{equation} \newline{}
The set of all modular forms of a fixed weight of $k$ forms a vector space, denoted by $M_k$. The set of modular forms belonging to $M_k$ that have a constant term of zero is denoted by $C_k$. For the details, see \cite[Chapter 14]{IK1}.
The $n$th Hecke operator $T_n$ on $C_k$ is defined by
$$
\left(T_n f\right)(z):=\frac{1}{n} \sum_{a d=n} a^k \sum_{b(\bmod d)} f\left(\frac{a z+b}{d}\right).
$$
\newline{}
It is well-known that there exists an orthonormal basis $S_k$ of $C_k$, consisting of eigenfunctions of all the Hecke operators $T_n$. These eigenfunctions are called  Hecke cusp forms. 
If $f$ is a Hecke cusp form, its eigenvalues  $\lambda_{f}(n)$ of the $n$th Hecke operator satisfy 
$$a(n)=a(1)\lambda_{f}(n),$$ and the function $n \mapsto \lambda_{f}(n)$ is multiplicative and real valued. 
For any Hecke cusp form $f,$ it is known that 
\begin{equation}
\sum_{n=1}^{X} \lambda_{f}(n) \ll_{f,\epsilon} X^{\frac{71}{192}+\epsilon} 
\end{equation}
for sufficiently large $X$ (see \cite{MJM}).
Therefore, when $x \in [X,2X],$ we have 
\begin{equation}\label{14} \begin{split} \sum_{n=x}^{x+h} \lambda_{f}(n) \ll \left|\sum_{n=1}^{x+h} \lambda_{f}(n)\right|+\left|\sum_{n=1}^{x} \lambda_{f}(n)\right| \ll_{f,\epsilon} x^{\frac{71}{192}+\epsilon} \end{split} \end{equation} for $1 \leq h \leq X.$
\newline{}
Moreover, for any small fixed $\epsilon > 0$ and $h \gg X^{\epsilon}$, we can apply Shiu's theorem (see Lemma \ref{Shiutheorem}) to deduce that 
\begin{equation}\label{15} \sum_{n=x}^{x+h} \lambda_{f}(n) \ll \sum_{n=x}^{x+h} |\lambda_{f}(n)| \ll h \prod_{p=1 \atop p \in \mathbf{P}}^{x}\left(1+\frac{|\lambda_{f}(p)|-1}{p}\right)
\end{equation}
where $\mathbf{P}$ is the set of all primes.
According to the Sato-Tate law, we have $$\prod_{p=1 \atop 
 p \in \mathbf{P} }^{X} \left(1+\frac{|\lambda_{f}(p)|-1}{p} \right) \ll (\log X)^{\frac{8}{3\pi}-1}$$ for sufficiently large $X$ (see \cite{Guangshi}, \cite{odoni}). Note that $\frac{8}{3\pi}-1 \sim -0.1511...$.  
\newline{}
In \cite{JUTILA1}, Jutila proved that for any fixed $\epsilon>0$ and $1 \leq l \leq X^{2/3},$ 
\begin{equation}\label{16} \sum_{n=1}^{X} \lambda_{f}(n) \lambda_{f}(n+l) \ll_{f,\epsilon} X^{2/3+\epsilon}. \end{equation}
We now apply the aforementioned bounds to get the following theorem.
\begin{Th}\label{th12}
Let $\epsilon>0$ be small fixed. Let $h,X$ be integers such that $1 \leq  h \ll X^{2/3}.$  
Then $$\frac{1}{h} \left|\sum_{n=x}^{x+h} \lambda_{f}(n)\right| \ll_{f,\epsilon} \max\Big(h^{-\frac{1}{2}}, X^{-\frac{1}{6}+\epsilon} \Big)$$ for almost all $x \in [X,2X],$ indicating that the proportion of integers in $[X,2X]$ that satisfy the given inequality approaches $1$ as $X$ tends to infinity. 
\end{Th}
\begin{proof} To apply the Chebyshev inequality, let us consider \begin{equation}\label{17} \frac{1}{X}\sum_{x=X}^{2X} \left|\frac{1}{h} \sum_{n=x}^{x+h} \lambda_{f}(n)\right|^{2}.\end{equation}
By squaring out the inner sums, we can represent \eqref{17} as 

\begin{equation} \frac{1}{Xh^{2}}\sum_{x=X}^{2X}\sum_{l=0}^{h}(2-\delta_{0}(l)) \sum_{n=x}^{x+h-l} \lambda_{f}(n) \lambda_{f}(n+l)  \nonumber\end{equation} where $\delta_{0}(l)=1$ for $l=0$ and $0$ otherwise. 
The diagonal terms ($l=0$) in \eqref{17} contribute
\begin{equation}\label{18}\frac{1}{Xh^{2}} \sum_{x=X}^{2X} \sum_{n=x}^{x+h} \lambda_{f}(n)^{2}.\end{equation} By using the fact that 
$\sum_{n=1}^{x} \lambda_{f}(n)^{2} \sim c_{f}X$ for some constant $c_{f}$ (see \cite[(14.56)]{IK1}), \eqref{18} is bounded by $O_{f}(h^{-1}).$ The off-diagonal terms ($1 \leq l \leq h$) in \eqref{17} contribute
\begin{equation}\label{19}\frac{1}{Xh^{2}} \sum_{l=1}^{h}\sum_{x=X}^{2X} \sum_{n=x}^{x+h-l} \lambda_{f}(n) \lambda_{f}(n+l).\end{equation}
 Therefore, by applying Deligne's bound $\lambda_{f}(n)\ll n^{\epsilon},$ \eqref{19} is bounded by 
\begin{equation}
 \frac{1}{Xh^{2}} \left(\sum_{l=1}^{h} (h-l) \sum_{n=X+h-l}^{2X-h+l} \lambda_{f}(n)\lambda_{f}(n+l) + \sum_{l=1}^{h} (h-l)^{2}X^{\epsilon}\right).
\nonumber\end{equation}
By using \eqref{16}, we have 
\begin{equation} \sum_{l=1}^{h} (h-l) \sum_{x=X+h-l}^{2X-h+l} \lambda_{f}(n)\lambda_{f}(n+l) \ll_{f,\epsilon} h^{2} X^{2/3+\epsilon}. \end{equation}
Therefore, \eqref{19} is bounded by 
\begin{equation} O_{f,\epsilon}\Big(X^{-1/3+\epsilon}+hX^{\epsilon-1}\Big).\end{equation}
By the assumption on the range of $h,$ \eqref{18} is bounded by 
$$O_{f,\epsilon} \max\Big(h^{-1}, X^{-1/3+\epsilon}\Big).$$ By applying the Chebyshev inequality, the proof is completed.
\end{proof}
\noindent
Note that the upper bound in Theorem \ref{th12} is better than the trivial bound provided by \eqref{15} when $h$ exceeds a certain power of $\log X,$ and is better than the trivial bound provided by \eqref{14} when $h\ll_{\epsilon} X^{\frac{103}{192}+\epsilon}.$  It's worth emphasizing that the upper bound for the shifted sums played a pivotal role in the proof of Theorem \ref{th12}
\subsection{Hecke-Maass cusp forms for $\bf{SL(n, \mathbb{Z})}$}
Let $\hbar^{n}=GL(n,\mathbb{R}) \slash (O(n,\mathbb{R}) \cdot  \mathbb{R}^{\times})$, where $\mathbb{R}^{\times}$ is the set of all invertible elements in $\mathbb{R}$. For any element $z$ in $\hbar^{n}$, by the Iwasawa decomposition (see \cite[Proposition 1.2.6]{Go}), it can be represented as $z=x \times y$, where 
\begin{equation} x=
\begin{bmatrix} 
	1 & x_{1,2} & x_{1,3} & \cdots   & x_{1,n}   \\
	 & 1 &x_{2.3}  &\cdots & x_{2,n}                     \\
	 &  &  \ddots & & \vdots                     \\
         &   &  & 1 &  x_{n-1,n}\\
        &   &  &   &  1 
	\end{bmatrix}, y=
\begin{bmatrix} 
	y_{1}y_{2}...y_{n-1} &  &    &   \\
	 & y_{1}...y_{n-2} &  & &                    \\
	 &  &  \ddots & &                   \\
         &   &  & y_{1} &  \\
        &   &  &   &  1 
	\end{bmatrix}
\end{equation}
with $x_{i,j} \in \mathbb{R}$\; and  \;$y_{i}>0.$
Let $ v=(v_{1},v_{2},...,v_{n-1}) \in \mathbb{C}^{n-1}$ and $m=(m_{1},m_{2}, ... , m_{n-2}, m_{n-1}) \in \mathbb{Z}^{n-1}.$
Let   
\begin{equation}
M(m_{1},m_{2},...,m_{n-1}):= 
\begin{bmatrix} 
	m_{1}m_{2}\cdots m_{n-2} |m_{n-1}|  &  &  &  &   \\
	 & \ddots &  & &                      \\
	 &  &  m_{1}m_{2} & &                      \\
         &   &  & m_{1} & \\
         &  & &  & 1  
	\end{bmatrix},
\end{equation}
and let $I_{v}(z)$ be defined by $$I_{v}(z):=\prod_{i=1}^{n-1} \prod_{j=1}^{n-1} y_{i}^{b_{i,j} v_{j}}.$$
For any $(n-1) \times (n-1)$ matrix $\gamma,$ define 
\begin{equation}
\gamma^{+}= 
\begin{bmatrix} 
	\gamma &     \\
	  &  1                  
	\end{bmatrix}.
\end{equation}
For any ring $\mathcal{B},$ let $U_{n}(\mathcal{B})$ be the set of all upper triangular matrices of the form
\begin{equation}
\begin{bmatrix} 
	I_{r_{1}} &  &  &  &   \\
	 & I_{r_{2}} &  &\ast &                      \\
	 &  &  \ddots & &                      \\
         &   &  &  &  I_{r_b}
	\end{bmatrix}
\end{equation}
for some natural numbers $r_1, r_2, \ldots, r_b$ such that $\sum_{i=1}^b r_i=n$, where $I_{r_i}$ denotes the $r_i \times r_i$ identity matrix, and $*$ denotes any arbitrary element in $\mathcal{B}$.
\noindent
Let $ \psi_{m}$ be the character of $U_{n}(\mathbb{R})$ defined by
$ \psi_{m}(u):= e^{2\pi i(m_{1}u_{1,2}+m_{2}u_{2,3}+...+m_{n-1}u_{n-1,n})}$ where
\begin{equation}
u= 
\begin{bmatrix} 
	1 & u_{1,2} & u_{1,3} & \cdots   & u_{1,n}   \\
	 & 1 &u_{2.3}  &\cdots & u_{2,n}                     \\
	 &  &  \ddots & & \vdots                     \\
         &   &  & 1 &  u_{n-1,n}\\
        &   &  &   &  1 
	\end{bmatrix}.
\end{equation}
For $z \in \hbar^{n}$ and $m \neq (0,0,...,0),$ the Jacquet-Whittaker function $W_{Jacquet}(z;v,\psi_{m})$ is defined by
\begin{equation}
W_{Jacquet}(z;v,\psi_{m})= \int_{U_{n}(R)} I_{v}(w_{n} \cdot u \cdot z) \overline{\psi_{m}(u)}d^{\ast}u,
\end{equation}
where
\begin{equation}
w_{n}= 
\begin{bmatrix} 
	  &  &  &  & (-1)^{[\frac{n}{2}]}  \\
	 &  &  & 1&                      \\
	 & \udots &  &  &                      \\
         1&   &  &  & 
	\end{bmatrix}. \nonumber
\end{equation}
Let $F(z)$ be a Hecke-Maass form for $SL(n, \mathbb{Z})$ given by 
\begin{equation}\begin{split} F(z)= \sum_{ \gamma \in U_{n-1}(\mathbb{Z})  \diagdown SL(n-1,\mathbb{Z})} & \sum_{m_{1}=1}^{\infty}  \cdots \sum_{m_{n-2}=1}^{\infty} \sum_{m_{n-1} \neq 0} \frac{\mathcal{A}_{F}(m_{1},m_{2},...,m_{n-1})}{\prod_{k=1}^{n-1} |m_{k}|^{\frac{k(n-k)}{2}}} 
\\& \times W_{Jacquet} \left(M\gamma^{+}z; v, \psi_{1,1,...,1,\frac{m_{n-1}}{|m_{n-1}|}}\right) \end{split} \end{equation}
(when $F(z)$ is a non-zero Hecke-Maass cusp form of type $v=\left(v_1, v_2, \ldots, v_{n-1}\right) \in \mathbb{C}^{n-1}$ for $S L(n, \mathbb{Z})$, we have the aforementioned expansion. see \cite[Theorem 9.3.11]{Go}). We have the following multiplicative relations 
$$  \mathcal{A}_{F}(m,1,...,1)\mathcal{A}_{F}(m_{1},m_{2},...,m_{n-1})= \sum_{\prod_{l=1}^{n} c_{l}=m \atop  c_{i}|m_{i} \textrm{for all} i \in [1,n-1]} \mathcal{A}_{F}\left(\frac{m_{1}c_{n}}{c_{1}}, \frac{m_{2}c_{1}}{c_{2}},...,\frac{m_{n-1}c_{n-2}}{c_{n-1}}\right).$$
For details, we refer the reader to \cite[Chapter 5, Chapter 9]{Go}. 
\subsection{The generalized Ramanujan-Petersson conjecture }\label{GRP2}
In order to prove that $\mathcal{A}_{F}(m,1,...,1)$ is divisor bounded (see Lemma \ref{Lemma 3.1}), it is necessary to assume the generalized Ramanujan-Petersson conjecture.  $L_{F}(s)$ has an Euler product
\begin{equation}\label{37} \begin{split}
L_{F}(s)= \prod_{p \in \bf{P}} &\Big(1-\mathcal{A}_{F}(p,...,1)p^{-s}+\mathcal{A}_{F}(1,p,...,1)p^{-2s} \\&- \cdots +(-1)^{n-1} \mathcal{A}_{F}(1,...,1,p)p^{-(n-1)s}+(-1)^{n}p^{-ns} \Big)^{-1}
\\&=\prod_{p \in \bf{P}} \prod_{j=1}^{n} \Big(1-a_{F,j}(p)p^{-s}\Big)^{-1}
\end{split}\end{equation}
for some $a_{F,j}(p) \in \mathbb{C}.$ The generalized Ramanujan-Petersson conjecture claims \begin{equation}\label{38} |a_{F,j}(p)|=1.\end{equation}
To study short averages of $\mathcal{A}_{F}(m,1,...,1),$ we will utilize the following result, which provides the upper bound for the long average of these values.
\begin{Lemma}[\cite{MJM}]
For sufficiently large $X,$
\begin{equation}\label{342} \sum_{m=1}^{X} \mathcal{A}_{F}(m,1,...,1) \ll_{F,\epsilon} X^{\frac{n^{2}-n}{n^{2}+1}+\epsilon}. \end{equation}
\end{Lemma}

\subsection{Zero-free regions and shifted convolution sum}
To establish results analogous to Theorem \ref{th12} on $\mathcal{A}_{F}(m,1,...,1)$, it suffices to obtain certain upper bounds for the following shifted sums 
$$\sum_{m=1}^{X} \mathcal{A}_{F}(m,1,...,1)\overline{\mathcal{A}_{F}(m+l,1,...,1)}$$ for $l \in [1,X]$ where $\overline{\mathcal{A}_{F}(m+l,1,...,1)}$ is the complex conjugate of $\mathcal{A}_{F}(m+l,1,...,1).$ Obtaining nontrivial upper bounds or asymptotics of the above shifted sum has many applications. As far as the author is aware, such results do not yet exist in the literature when $n \geq 3.$ As an alternative, let us consider successful approaches in multiplicative number theory.
Let $s=\sigma+it\thinspace$ where $\sigma, t \in \mathbb{R}.$ Define 
$$L_{F}(s):=\sum_{m=1}^{\infty} \frac{\mathcal{A}_{F}(m,1,...,1)}{m^{s}}.$$
The generalized Riemann hypothesis claims that there are no nontrivial zeros of $L_{F}(s)$ for $\sigma>\frac{1}{2}.$ The current state of knowledge of the zero-free region is, there are no nontrivial zeros of $L_{F}(s)$ where $$\sigma > 1- \frac{c_{F}}{\log |t+3|}$$ where $c_{F}>0$ is a constant depending on the conductor of $F$. In contrast, for the Riemann zeta function $\zeta(s),$ it is known that there are no nontrivial zeros where  \begin{equation} \sigma> 1- \frac{c_{1}}{(\log |t+3|)^{2/3}\log \log |t+3|} \nonumber \end{equation} and $|t|\gg 1$  for some constant $c_{1}>0.$
This zero-free region, characterized by the power of the logarithm in the denominator is $2/3$, is called the Vinogradov-Korobov zero-free region. Using this zero-free region, it can be shown that for any $A>0$ and any small fixed $\epsilon>0,$ we have 
\begin{equation}\label{goodin}\sum_{p=P}^{Q} \frac{1}{p^{1+it}} \ll_{A}  (\log X)^{-A} \quad \textrm{when} \quad t > (\log X)^{A+2}\end{equation}
for $e^{(\log X)^{2/3+\epsilon}} \ll P \leq Q \ll X$ (see \cite[Lemma 2]{Matomaki1}). By applying the above inequality, the authors of \cite{Matomaki1} show that nontrivial results concerning the  Liouville function in short interval can be obtained by using methods that are relatively simpler than the methods presented in \cite{MR3} (note that the methods outlined in \cite{MR3} are sophisticated and comprehensive, demonstrating improved results in cases where broad zero-free regions are not available). This motivates us to prove an analog of \eqref{goodin}, under the assumption of the generalized Ramanujan-Petersson conjecture (see \eqref{38}) and the existence of $c_{F}$ such that $L_{F}(s)$ satisfies the following zero-free region (we will refer to this zero-free region as the $SL(n,\mathbb{Z})$ Vinogradov-Korobov zero-free region):
\begin{equation}\label{Zerofree}\sigma> 1- \frac{c_{F}}{(\log |t+3|)^{2/3}\log \log |t+3|}.\end{equation} 
 
\begin{Rem}
Let $L(s,\text{sym}^{2}f_{j})$ be the symmetric square L-function of $j$th  Hecke-Maass cusp forms $f_{j}$ (with the Laplacian eigenvalue $\lambda_{j}:=\frac{1}{4}+t_{j}^{2}, t_{j}>0$) for $SL(2,\mathbb{Z})$ (for details, see \cite[Chapter 15]{IK1}). Define $$J =\left\{j:t_j \leqq T \log ^2 T, D_j(s) \text { has a zero in } \sigma \geqq 1-\eta,|t| \leqq \log ^3 T\right\} $$ where 
 $L(s, \text{sym}^2f_j)=\zeta(2 s) D_j(s).$ 
Luo \cite{LuoWZ1999} proved that 
$$\left|J\right| \ll T^{1 / 5}$$ if $\eta$ is sufficiently small.
A direct application of \cite[Theorem 1.2]{ThornerZaman} can show that almost all L-functions associated to Hecke-Maass cusp forms for $SL(n,\mathbb{Z})$ (with some restrictions on the size of analytic conductor) has the zero-free region
$\sigma > 1- \alpha, t \gg 1$ for sufficiently small $\alpha>0.$ 
\end{Rem}

\begin{Lemma}\label{Lemma12} Let $\epsilon>0$ be small fixed, and let $A$ be sufficiently large. Assume the zero-free region \eqref{Zerofree} and the generalized Ramanujan-Petersson conjecture. Then for  $e^{(\log X)^{2\epsilon}} \ll t \ll X,$ we have   
\begin{equation}\label{113} \sum_{p=P_{1}\atop p \in \mathbf {P}}^{Q_{1}} \frac{\mathcal{A}_{F}(p,1,...,1)}{p^{1+it}} \ll_{F} (\log X)^{-A}  \end{equation} 
for all $e^{ (\log X)^{2/3+\epsilon}}<P_{1}<Q_{1}<e^{ (\log X)^{1-\epsilon}}.$
\end{Lemma}
\begin{proof}
The proof of this basically follows from the proof of \cite[Lemma 2]{Matomaki1}.
Let $\kappa= \frac{1}{\log X}, T= \frac{|t|+1}{2}.$ 
By using Perron's Formula, we have 
\begin{equation}\label{114}\begin{split}
\sum_{p=P_{1}\atop p \in \mathbf {P}}^{Q_{1}} \frac{\mathcal{A}_{F}(p,1,...,1)}{p^{1+it}} = \frac{1}{2\pi i} & \int_{\kappa-iT}^{\kappa+iT} \log L_{F}(s+1+it) \frac{Q_{1}^{s}-P_{1}^{s}}{s} ds + O(R_{F})\\+O_{\epsilon}(P_{1}^{-1+\epsilon}) \end{split}\end{equation}
where $$R_{F} = \max \left(\sum_{m=1}^{4Q_{1}} \frac{|\mathcal{A}_{F}(m,1,...,1)|}{m\max \left(1, T\left|\log \frac{Q_{1}}{m}\right|\right)}\left(\frac{Q_{1}}{m}\right)^{\kappa} ,\sum_{m=1}^{4P_{1}} \frac{\left|\mathcal{A}_{F}(m,1,...,1)\right|}{m\max \left(1, T\left|\log \frac{P_{1}}{m}\right|\right)}\left(\frac{P_{1}}{m}\right)^{\kappa}\right) $$ (see \cite[Corollary 5.3]{Montgomery2}). Note that $O_{\epsilon}\left(P_{1}^{-1+\epsilon}\right)$ arises from a simple application of Lemma \ref{Lemma 3.1} ($\mathcal{A}_{F}(p^{k},1,...,1) \ll d_{m}(p^{k})$) to 
\begin{equation}\sum_{p^{k}, \atop k \geq 2 , p \in [P_{1},Q_{1}]\cap \mathbf{P}} \frac{|\mathcal{A}_{F}\left(p^{k},1,...,1\right)|}{p^{k(1+it)}}. \nonumber \end{equation}
By applying the upper bound \eqref{314} ($\mathcal{A}_{F}\left(m,1,...,1\right) \ll_{\epsilon} m^{\epsilon}$),  it is straightforward to show that $R_{F} \ll_{F.\epsilon} Q_{1}^{\epsilon}/T.$
Let $\sigma_{0}=\frac{c_{F}/2}{(\log X)^{2/3}(\log \log X)^{1/3}}.$
By shifting the line of integration in \eqref{114} to $[-\sigma_{0}-iT, -\sigma_{0}+iT]$ and changing variables, we have 
\begin{equation}\label{115}\begin{split}
\sum_{p=P_{1},\atop p \in \mathbf {P}}^{Q_{1}} \frac{\mathcal{A}_{F}(p,1,...,1)}{p^{1+it}} &= \frac{1}{2\pi i} \int_{-iT}^{+iT} \log L_{F}(1-\sigma_{0}+it+iu) \frac{Q_{1}^{-\sigma_{0}+iu}-P_{1}^{-\sigma_{0}+iu}}{-\sigma_{0}+iu} du \\&+ O_{F,\epsilon}\left(\frac{Q_{1}^{\epsilon}}{T}\right)+O_{F}\left(P_{1}^{-1}\right). \end{split}\end{equation}
Note that a standard application of the Borel-Caratheodory theorem gives (see \cite[Section 4]{FG}),  \begin{equation}\label{116} \log L_{F}(\sigma_{1}+it+iu) \ll_{F} |\log (t+u+2)|\end{equation} for $\sigma_{1} \in [1-\sigma_{0}, 1+\kappa].$
Therefore, by applying \eqref{116} to \eqref{115}, \eqref{113} is bounded by 
$(\log X)^{-A}.$
\end{proof}
\begin{Rem} Without assuming the zero-free region, the argument presented in the proof of Lemma \ref{Lemma12} gives that for $f_{j}$ where $t_{j} \leq T\log^{2}T$ and $ (\log T)^{3}>X,$
\begin{equation}\label{symsquare}\sum_{p=P_{1}}^{Q_{1}} \frac{1+\lambda_{f_{j}}(p^{2})}{p^{1+it}} \ll_{f_{j}} (\log X)^{-A}
\end{equation}
holds for almost all $j \in J.$ Additionally, it can be directly deduced that almost all L-functions associated to Hecke-Maass cusp forms for $SL(n,\mathbb{Z})$  satisfy \eqref{113}.
\end{Rem}
\noindent In \cite{Jesse}, J\"{a}\"{a}saari showed that under the assumption of the generalized Lindel\"{o}f hypothesis and a weak version of the Ramanujan-Petersson conjecture,
when $F$ is a nontrivial Hecke-Maass cusp form for $SL(n,\mathbb{Z}),$ 
\begin{equation}
\frac{1}{X} \int_{X}^{2X} \Big| \sum_{m=x}^{x+h} \mathcal{A}_{F}(m,1,...,1) \Big|^{2} dx \sim B_{F} X^{1-\frac{1}{n}}  
\nonumber
\end{equation} holds for some constant $B_{F},$ where $h \in [X^{1-\frac{1}{n}+\epsilon}, X^{1-\epsilon}]$ for some small fixed $\epsilon>0.$
In the following subsection, we consider the wider range $h \in [e^{(\log X)^{1-\epsilon}}, X^{1-\epsilon}].$
\subsection{Main results}

\begin{Theorem-non} Let $\epsilon>0$ be small fixed, and let $X$ be sufficiently large. Let $e^{ (\log X)^{1-\epsilon}} \ll_{\epsilon} h\ll_{\epsilon} X^{1-\epsilon}.$   Assume the generalized Ramanujan-Petersson conjecture on $F$ and the $SL(n,\mathbb{Z})$ Vinogradov-Korobov zero-free region \eqref{Zerofree}. Then

\begin{equation}
\frac{1}{X} \int_{X}^{2X} \left| \sum_{m=x}^{x+h} \mathcal{A}_{F}(m,1,...,1) \right|^{2} dx \ll_{F,\epsilon} h^{2}(\log X)^{-2/3+\epsilon}.
\nonumber
\end{equation} 
\end{Theorem-non}
\noindent This theorem implies the following corollaries.
\begin{Cor}\label{Theorem 3.10} Let $\epsilon>0$ be small fixed, and let $X$ be sufficiently large. Let $e^{ (\log X)^{1-\epsilon}} \ll_{\epsilon} h\ll_{\epsilon} X^{1-\epsilon}.$   Assume the generalized Ramanujan-Petersson conjecture on $F$ and the $SL(n,\mathbb{Z})$ Vinogradov-Korobov zero-free region \eqref{Zerofree}. Then 
 
\begin{equation}\label{grpgvk} \frac{1}{h}\sum_{m=x}^{x+h} \mathcal{A}_{F}(m,1,...,1) \ll_{F,\epsilon}  (\log X)^{-\frac{1}{3}+\epsilon} \end{equation}
for almost all $x\in [X,2X].$
\end{Cor}

\begin{Cor}\label{Cor2}
Let $\epsilon>0$ be small fixed, and let $X$ be sufficiently large. Let $e^{ (\log X)^{1-\epsilon}} \ll_{\epsilon} h\ll_{\epsilon} X^{1-\epsilon}.$   Assume the generalized Ramanujan-Petersson conjecture on $F$ and the zero-free region \eqref{Zerofree}. Then
\begin{equation} \frac{1}{Xh^{2}}\sum_{l=-h \atop l \neq 0}^{h} |h-l|  \left(\sum_{m=X}^{2X}\mathcal{A}_{F}(m,1,...,1) \overline{\mathcal{A}_{F}(m+l,1,...,1)} \right) \ll_{F,\epsilon}  (\log X)^{-2/3+\epsilon}. \nonumber\end{equation}
\end{Cor}
\noindent The $SL(n,\mathbb{Z})$ Sato-Tate conjecture, coupled with the generalized Ramanujan-Petersson conjecture, implies that for sufficiently large $X,$
\begin{equation}\label{Sato}
\prod_{p \leq X} \left(1+\frac{|\mathcal{A}_{F}(p,1,...,1)|-1}{p}\right) \asymp (\log X)^{\kappa_{n}-1}    
\end{equation} 
where $$\kappa_{n}:= \frac{1}{n! (2 \pi)^{n-1}}\int_{[0,2\pi)^{n-1} \atop \theta_{1}+\theta_{2}+...+\theta_{n}=0} \left|e^{i\theta_{1}}+e^{i\theta_{2}}+...+e^{i\theta_{n}}\right| \prod_{1 \leq j < k \leq n} \left|e^{i\theta_{j}}-e^{i\theta_{k}}\right|^{2}d\theta_{1}d\theta_{2}...d\theta_{n-1}$$ (see \cite[(1.16)]{YGZ}). For example, $\kappa_{3} \asymp 0.8911, \kappa_{4} \asymp 0.8853,$ and $\kappa_{5} \asymp 0.8863.$ 
Therefore, by using \eqref{shiu2} and \eqref{Sato}, we have
\begin{equation}\label{Trivial} \frac{1}{h}\sum_{l=1}^{h}\sum_{m=X}^{2X} \mathcal{A}_{F}(m,1,....,1)\overline{\mathcal{A}_{F}(m+h,1,...,1)} \ll X(\log X)^{2\kappa_{n}-2} \end{equation} under the generalized Ramanujan-Petersson conjecture. By assuming the $SL(n,\mathbb{Z})$ Vinogradov-Korobov zero-free region instead of the Sato-Tate conjecture, when $\kappa_{n} > 5/6,$ we can obtain a superior upper bound for the shifted sums compared to  \eqref{Trivial}. 
\begin{Cor}\label{Cor13} Let $\epsilon>0$ be small fixed and let $X$ be sufficiently large. Let $e^{ (\log X)^{1-\epsilon}} \ll_{\epsilon} h\ll_{\epsilon} X^{1-\epsilon}.$ Assume the generalized Ramanujan-Petersson conjecture on $F$ and the zero-free region \eqref{Zerofree}. Then
\begin{equation}\label{Shiftedsum} \frac{1}{X(2h-1)}\sum_{l=-h \atop l \neq 0}^{h}  \left(\sum_{m=X}^{2X}\mathcal{A}_{F}(m,1,...,1) \overline{\mathcal{A}_{F}(m+l,1,...,1)} \right) \ll_{F,\epsilon}  (\log X)^{-1/3+\epsilon}. \end{equation} 
\end{Cor}
\begin{proof} 
The left-hand side of \eqref{Shiftedsum} can be represented as 
\begin{equation}\frac{1}{X}\int_{X}^{2X} \left(\frac{1}{2h-1} \sum_{m=[x-h] \atop m \neq x}^{x+h} \overline{\mathcal{A}_{F}(m,1,...1)}\right) \mathcal{A}_{F}([x],1,...,1) dx. \end{equation}
Note that
\begin{equation}\label{Rankinselbergappl} \int_{X}^{2X} \left|\mathcal{A}_{F}([x],1,...,1)\right|^{2} dx \ll \sum_{n=X}^{2X} \left|\mathcal{A}_{F}(n,1,...,1)\right|^{2} \ll_{F} X \end{equation}
(see \cite[Remark 12.1.8]{Go}).
By applying H\"older's inequality, Theorem \ref{Theorem 3.10}, and \eqref{Rankinselbergappl}, we see that 
\begin{equation}\begin{split}
&\int_{X}^{2X} \left(\frac{1}{2h-1} \sum_{m=[x-h] \atop m \neq x}^{x+h} \overline{\mathcal{A}_{F}(m,1,...1)}\right) \mathcal{A}_{F}([x],1,...,1) dx
\\&\ll_{F} \left(\int_{X}^{2X} \left|\frac{1}{2h-1} \sum_{m=[x-h]}^{x+h} \mathcal{A}_{F}(m,1,...1)\right|^{2}  dx \int_{X}^{2X}    |\mathcal{A}_{F}([x],1,...,1)|^{2} dx  \right)^{1/2}   
\\&\ll_{F,\epsilon} X ( \log X)^{-1/3+\epsilon}.
\end{split} 
\end{equation}
\end{proof}
\begin{Rem}\label{Remark15}
If the zero-free region  \begin{equation}\label{anotherzero} \sigma> 1- \frac{c_{1}}{(\log |t|+3)^{\theta}} \end{equation} holds for some constants $c_{F}>0,$ $\theta <1,$ then it is easy to show that we can replace $ (\log X)^{-1/3+\epsilon}$ in Theorem \ref{Theorem 3.10} with $ (\log X)^{\theta-1+\epsilon}$ (the cancellation $\log P/ \log Q$ comes from the truncation in Lemma \ref{Lemma34}. Therefore, we only need to replace $P=e^{(\log X)^{2/3+\epsilon}}$ with $P=e^{(\log X)^{\theta+\epsilon}}$ in the beginning of Section 2). Since the Vinogradov-Korobov zero-free region is proved for the Riemann zeta function, one could get similar cancellations on the symmetric square L-functions. For example, without assuming the zero-free region, one may apply \eqref{symsquare}. 
\end{Rem}
 We say that a multiplicative function $f:\mathbb{N} \rightarrow \mathbb{R}$ has $k$ sign changes in $[1,x]$ if there exist integers $1 \leq m_{1} < m_{2} <... < m_{k+1} \leq x$ such that $\prod_{i=1}^{k+1}f(m_{i})\neq 0$ and $f(m_{i})f(m_{i+1})<0$ for all $i \leq k.$ In  \cite{MR}, it has been demonstrated that when $f(n)$ is negative for some $n$, and $f(n)$ is non-zero for a positive proportion, then, as $\psi(x)$ tends toward infinity, almost every interval $[x, x+\psi(x)]$ has a sign change (see \cite[Corollary 4]{MR}.
As an application of \eqref{anotherzero}, let us talk about the sign changes of $\{\mathcal{A}_{F}(m,1,...,1)\}_{m=1}^{X}.$ Since we do not know that $\mathcal{A}_{F}(m,1,...,1)$ is non-zero for a positive proportion, we cannot apply the aforementioned result. 

\noindent By assuming the Ramanujan-Petersson conjecture, it can be proved that \begin{equation}\label{1122}\sum_{m \leq x} |\mathcal{A}_{F}(m,1,...1)| \gg x (\log x)^{\frac{1}{n}-1}.\end{equation} 
This lower bound can be derived from the work of Wirsing in \cite{Wirsing} along with the following lower bound:  
\begin{equation}\begin{split} \sum_{p \leq x} |\mathcal{A}_{F}(p,1,...,1)|\log p &\geq \frac{1}{n}\sum_{p \leq x} |\mathcal{A}_{F}(p,1,...,1)|^{2}\log p \\&= \frac{1}{n}\big(1+o(1)\big)x
\end{split}\end{equation}
(see \cite[(1.15)]{YGZ}). 
When $F$ is self-dual, $\mathcal{A}_{F}(m,1,...,1) \in \mathbb{R}$ for all $m \in \mathbb{N}.$ Hence, for the self-dual Hecke-Maass cusp forms $F,$  if
$$\sum_{m=x}^{x+h} \mathcal{A}_{F}(m,1,...,1) = o\left( \sum_{m=x}^{x+h} |\mathcal{A}_{F}(m,1,...,1)| \right)$$ and
$$\sum_{m=x}^{x+h} |\mathcal{A}_{F}(m,1,...,1)| \neq 0, $$
then 
$$ \sum_{m=x}^{x+h} \left(|\mathcal{A}_{F}(m,1,...,1)| \pm  \mathcal{A}_{F}(m,1,...,1)\right) >0, $$ which indicates that $[x, x+h]$ has a sign change.  
 Therefore, we will prove the following corollary in Section 3.
\begin{Cor}\label{Cor16} Let $\epsilon>0$ be sufficiently small and let $X$ be sufficiently large. Let $F$ be a self-dual Hecke-Maass cusp form for $SL(n,\mathbb{Z}).$  Assume the generalized Ramanujan-Petersson conjecture on $F,$ the zero-free region \eqref{anotherzero} for some $\theta< \frac{1}{n}.$  Then the number of sign changes in  $\{\mathcal{A}_{F}(m,1,...,1)\}_{m=1}^{X}$
is $\gg_{F,\epsilon} \frac{X}{e^{(\log X)^{2-\epsilon}}}.$ 
\end{Cor}

\begin{Rem}We use the generalized Ramanujan-Petersson conjecture to give the upper bounds of second moments, shifted convolution sums of Hecke eigenvalue squares over some subsets in long intervals. 
\end{Rem}
\begin{Rem}
Applying methods from \cite{mangerel2021divisorbounded} can deduce results similar to our results. The methods in \cite{mangerel2021divisorbounded} rely on understanding the relationship between values of multiplicative functions on $p$ and Dirichlet characters $p^{it}$ over primes $p$ instead of assuming the $SL(n,\mathbb{Z})$ Vinogradov-Korobov zero-free regions.
The results from the above methods may hold for a wider range of $h$, but the resulting cancellations are weaker compared to our finding. Moreover, we lack a comprehensive understanding of the correlations between $\mathcal{A}_{F}(p,1,...,1)$ and $p^{it}.$ Specifically, under the generalized Ramanujan-Petersson conjecture, the best possible upper bound for the left-hand side of \eqref{grpgvk} from the methods (see \cite[Theorem 1.7]{mangerel2021divisorbounded}) is 
$$(\log X)^{-\frac{(0.1211...)}{32+\frac{42}{n}}}.$$ Note that the above upper bound is larger than \eqref{Trivial}. 
\end{Rem}

\section{Lemmas}
\noindent From now on, let $\epsilon$ be a small, fixed positive number, and the value of $\epsilon$ may differ from one occurrence to another. Let $P=e^{ (\log X)^{2/3+\epsilon}}$ and $Q=e^{ (\log X)^{1-\epsilon}}.$
Let $S_{d}$ be the set of integers in $\thinspace [X/d,2X/d] \thinspace$ having at least one prime factor in $[P,Q]$. We denote $S_{d} ':= \{n \in [X/d,2X/d]: n \notin S_{d}\}.$ For convenience we denote $S$ instead of $S_{1},$ $S'$ instead of $S_{1}'.$ For simplicity, we use $$\prod_{p} a_{p} \,,\; \sum_{p} a_{p}  \;\; \textrm{instead} \; \textrm{of} \;\; \prod_{p \in \mathbf{P}} a_{p} \,,\; \sum_{p \in \mathbf{P}}a_{p}$$ respectively, for any $a_{p}$. Let $d_{n}(m)$ be the $n$th divisor function (the coefficients of $m^{-s}$ of $\zeta(s)^{n}$). For convenience, let $\bf{GVK}$ stand for the $SL(n,\mathbb{Z})$ Vinogradov-Korobov zero-free region \eqref{Zerofree} and let $\bf{GRP}$ stand for the generalized Ramanujan-Petersson conjecture. We will apply the following lemma to obtain upper bounds for mean values of multiplicative functions over intervals.

\begin{Lemma}\label{Shiutheorem}\rm{(Variants of Shiu's theorem)}
Let $\alpha, \beta \geq 1,$ $\epsilon>0.$ Let $g$ be a non-negative multiplicative function such that 
\begin{equation}\label{22} \begin{split} &g(p^{v}) \leq \alpha^{v} \rm{\quad for \thinspace \thinspace all} \thinspace \mathit{p} \in \mathbf{P}, \mathit{v}\in \mathbb{N}, \\
                          &g(n) \leq \beta n^{\frac{\epsilon^{3}}{1000}} \rm{\quad for \thinspace \thinspace all \thinspace } \mathit{n} \in \mathbb{N}. \end{split} \end{equation}
Then for $2 \leq X^{\epsilon} \leq Y \leq X, 1\leq H \leq X,$
\begin{equation} \frac{1}{Y} \sum_{n=X}^{X+Y} g(n) \ll_{\alpha,\beta, \epsilon} \prod_{p=1}^{X} \left(1+\frac{ g(p)-1}{p}\right), \end{equation}

\begin{equation}\label{shiu2} \sum_{l=1}^{H} \sum_{n=X}^{X+Y} g(n)g(n+l) \ll_{\alpha, \beta, \epsilon} HY  \prod_{p=1}^{X} \left(1+\frac{ g(p)-1}{p}\right)^{2}. \end{equation}
\end{Lemma} 
\begin{proof}
See \cite[Lemma 2.3]{MRT2}.
\end{proof}
\noindent
By using Lemma \ref{Shiutheorem}, it is easy to check that typical numbers in $[X/d,2X/d]$ are contained in $S_{d}.$ 
By using \eqref{22}, we have 
\begin{equation} \begin{split}\Big| \{n \in [X,2X]: n \notin S_{d}\} \Big| &\ll X\log P/\log Q^{d} \\ &\ll X(\log X)^{-1/3+\epsilon}/d.\end{split} \end{equation} 

\begin{Lemma}\label{Lemma22} Let $h \geq 1.$ For sufficiently large $X,$
\begin{equation}\sum_{l=1}^{h}  \sum_{n=X}^{2X}1_{S'}(n)1_{S'}(n+l) \ll_{\epsilon}  hX (\log P / \log Q)^{2} \end{equation}
\end{Lemma}
\begin{proof}
This arises from a direct application of \eqref{shiu2} with  $g(n)=1_{S'}(n)$.
\end{proof}
\noindent
The following lemma will be needed after Lemma \ref{Lemma 3.6}.
\begin{Lemma}\label{Lemma23} \rm{(Dirichlet mean value theorem)} For any $\{a_{n}\}_{n \in N} \subset \mathbb{C},$ 
\begin{equation}\int_{0}^{T} \left| \sum_{n=X}^{2X} \frac{a_{n}}{n^{it}} \right|^{2}dt \ll (T+X) \sum_{n=X}^{2X} |a_{n}|^{2}. \end{equation}
\end{Lemma}
\begin{proof}
See \cite[Theorem 9.2]{Montgomery1}.
\end{proof}
\section{Proof of Theorems and corollaries}
The following lemma states that $\mathcal{A}_{F}(m,1,1...,1) $ satisfies the divisor bound, under $\bf{GRP}$.
\begin{Lemma}\label{Lemma 3.1} Assume $\bf{GRP}$.   Then for any $k \in \mathbb{N}$ and $p \in \mathbf{P},$ $$\mathcal{A}_{F}(p^{k},1,1...,1) \leq d_{n}(p^{k}).$$ \end{Lemma}
\begin{proof}
By the Euler product of $L_{F}(s),$
$$\mathcal{A}_{F}\left(p^k,1,1,...,1\right)=\sum_{\substack{r_1, \ldots, r_n \geq 0 \\ r_1+\cdots+r_n=k}} \prod_{1 \leq j \leq n} a_{F,j}(p)^{r_j}.$$
By using $\bf{GRP},$ we have
\begin{equation}\sum_{\substack{r_1, \ldots, r_n \geq 0 \\ r_1+\cdots+r_n=k}} \prod_{1 \leq j \leq n} a_{F,j}(p)^{r_j} 
\leq \sum_{\substack{r_1, \ldots, r_n \geq 0 \\ r_1+\cdots+r_n=k}} 1
\leq d_{n}(p^{k}).\nonumber \end{equation}
 \end{proof}
 Note that $\mathcal{A}_{F}(p,1,...,1) \leq d_{n}(p)=n.$
\noindent In order to apply Lemma \ref{Shiutheorem}, we also need to show that for arbitrarily small fixed $\epsilon>0,$
\begin{equation}\label{314} \mathcal{A}_{F}(m,1,...,1) \ll_{\epsilon} m^{\epsilon}. \end{equation}
By the Euler product \eqref{37} and $\bf{GRP},$ we see that
\begin{equation} \mathcal{A}_{F}(m,1,...,1) \ll d_{n}(m).  \nonumber\end{equation}
Given that $d_{n}(m) \leq d_{2}(m)^{n}$, and utilizing $d_{2}(m) \leq e^{(\log 2 +o(1))\frac{ \log m}{ \log \log m}}$, we can establish \eqref{314}.
By applying Lemma \ref{Shiutheorem} along with the bound $\mathcal{A}_{F}(p,1,...,1) \leq n,$ we have
\begin{equation}\label{315} \sum_{m=1}^{X} |\mathcal{A}_{F}(m,1,...,1)|^{4} \ll X (\log X)^{n^{4}-1}. \end{equation}
\subsection{Proof of theorem \ref{Theo1}}
This subsection is devoted to proving Theorem \ref{Theo1}.
\noindent In order to prove Theorem \ref{Theo1}, we need the following lemmas. 
\begin{Lemma}\label{Lemma33} For sufficiently large $X,$
\begin{equation}\label{317}\sum_{p=1}^{X} \frac{|\mathcal{A}_{F}(p,1,...,1)|^{2}}{p}= \sum_{p=1}^{X} \frac{1}{p} + O_{F}(1). \end{equation}
\end{Lemma}
\begin{proof}
The proof of this theorem is essentially in line with \cite[Theorem 5.13]{IK1}.
Let $\tilde{F}$ be the dual of $F.$ The Rankin-Selberg L-function is given by
\begin{equation} L_{F \times \tilde{F}}(s)=\sum_{n=1}^{\infty} \frac{|\mathcal{A}_{F}(n,1,...,1)|^{2}}{n^{s}}, \end{equation} and it is meromorphic with a simple pole at $s=1.$ Furthermore, there exists a constant $d_{F}>0$ such that 
\begin{equation} L_{F \times \tilde{F}}(\sigma+it) \neq 0\end{equation} for $\sigma >1-\frac{d_{F}}{(\log |t|+3)}$ with $|t|\gg 1$ (see \cite{HB}).
Therefore, by using \cite[Theorem 5.13]{IK1}, we have 
$$\sum_{p=1}^{X} |\mathcal{A}_{F}(p,1,...,1)|^{2} \Lambda(p) = x+ o_{F}(x)$$ 
where $\Lambda$ is the von-Mangoldt function. Hence, applying partial summation over $p$ gives \eqref{317}.

\end{proof}

\begin{Lemma}\label{Lemma34} Assume $\bf{GRP}$ and $\bf{GVK}$. Let $(\log X)^{n^{4}-1/3-\epsilon} \ll_{\epsilon} h =o(X).$ Then
\begin{equation}\label{320} \int_{X}^{2X} \left|\frac{1}{h} \sum_{m=x}^{x+h} \mathcal{A}_{F}(m,1,...,1)1_{S'}(m)\right|^{2} dx \ll_{F,\epsilon} X (\log X)^{-2/3+\epsilon}. \end{equation}
\end{Lemma}
\begin{proof}

By using the Cauchy-Schwarz inequality, we see that 
\begin{equation}\label{321} \begin{split} \int_{X}^{2X} & \left|\frac{\sum_{m=x}^{x+h} \mathcal{A}_{F}(m,1,...1)1_{S'}(m)}{h}\right|^{2} dx \\&\ll  \frac{1}{h^{2}} \int_{X}^{2X} \left|\sum_{m=x}^{x+h} 1_{S'}(m)\right| \left| \sum_{m=x}^{x+h} |\mathcal{A}_{F}(m,1,...1)1_{S'}(m)|^{2} \right| dx \\&\ll \frac{1}{h^{2}} \left( \int_{X}^{2X} \left( \sum_{m=x}^{x+h} 1_{S'}(m) \right)^{2} dx \right)^{\frac{1}{2}} \left( \int_{X}^{2X} \left( \sum_{m=x}^{x+h} |\mathcal{A}_{F}(m,1,...1)1_{S'}(m)|^{2} \right)^{2} dx \right)^{\frac{1}{2}}.
\end{split}\end{equation}
By squaring out the inner sum over $m$ and using Lemma \ref{Lemma22}, we have
\begin{equation} \begin{split} \int_{X}^{2X} \left( \sum_{m=x}^{x+h} 1_{S'}(m) \right)^{2} dx &\ll h \left(\sum_{l=1}^{h}\sum_{m=X}^{2X} 1_{S'}(m)1_{S'}(m+l) + \sum_{m=X}^{2X+h} 1_{S'}(m) \right) \\ &\ll h^{2}X \left(\frac{\log P}{\log Q}\right)^{2} +hX \frac{\log P}{\log Q}\\ &\ll_{\epsilon} h^{2}X (\log X)^{-2/3+2\epsilon}. \end{split} \end{equation}
In addition, by squaring out the last integrand in \eqref{321}, we see that 
\begin{equation}\begin{split}  \int_{X}^{2X} \Big(\sum_{m=x}^{x+h} &|\mathcal{A}_{F}(m,1,...1)1_{S'}(m)|^{2} \Big)^{2} dx \ll h\sum_{m=X}^{2X+h} |\mathcal{A}_{F}(m,1,...1)1_{S'}(m)|^{4} \\&+ h\sum_{l=1}^{h} \sum_{m=X}^{2X} |\mathcal{A}_{F}(m,1,...1)1_{S'}(m)|^{2} |\mathcal{A}_{F}(m+l,1,...,1)1_{S'}(m+l)|^{2}.\end{split}\end{equation}
By applying Lemma \ref{Shiutheorem} and \ref{Lemma33}, we have 
$$h\sum_{l=1}^{h} \sum_{m=X}^{2X} |\mathcal{A}_{F}(m,1,...1)1_{S'}(m)|^{2} |\mathcal{A}_{F}(m+l,1,...,1)1_{S'}(m+l)|^{2} \ll_{F,\epsilon} h^{2}(2X+h)(\log X)^{-2/3+2\epsilon}.$$
Hence, by using the assumption $(\log X)^{n^{4}-1/3-\epsilon} \ll h=o(X)$ and \eqref{315}, we see that 
$$  \int_{X}^{2X} \big| \sum_{m=x}^{x+h} \mathcal{A}_{F}(m,1,...1)^{2}1_{S'}(m) \big|^{2} dx \ll_{F,\epsilon} h^{2}X(\log X)^{-2/3+\epsilon}.$$ The proof is completed.
\end{proof}
\begin{Rem}In order to obtain results similar to \eqref{320}, one may avoid assuming $\bf{GRP}$ by employing certain large sieve techniques when $h$ is bigger than $X^{\epsilon}.$  However, the resulting upper bounds will be weaker than \eqref{320} (see \cite[Theorem 7.14]{IK1}). 
\end{Rem}

\begin{Lemma}\label{Lemma 3.6} \rm{(Parseval's bound)}
Let $X$ and $A$ be sufficiently large. Let $e^{(\log X)^{1-\epsilon}} \ll_{\epsilon} h_{1} \leq h_{2}=O_{\epsilon}(X^{1-\epsilon}).$ 
Then
\begin{equation}
\frac{1}{X} \int_{X}^{2X} \left|\frac{1}{h_{1}}\sum_{m=x}^{x+h_{1}} \mathcal{A}_{F}(m,1,...,1) 1_{S}(m) - \frac{1}{h_{2}}\sum_{m=x}^{x+h_{2}} \mathcal{A}_{F}(m,1,...,1) 1_{S}(m)  \right|^{2}dx \end{equation}
\begin{equation} \begin{split} \ll\max_{i=1.2}& \Big( \int_{ e^{(\log X)^{2\epsilon}}}^{Xh_{i}^{-1}} \Big|\sum_{m=X}^{2X} \frac{\mathcal{A}_{F}(m,1,...,1)1_{S}(m)}{m^{1+it}}\Big|^{2}dt +\\& \max_{T>Xh_{i}^{-1}} \frac{Xh_{i}^{-1}}{T} \int_{T}^{2T} \Big|\sum_{m=X}^{2X} \frac{\mathcal{A}_{F}(m,1,...,1) 1_{S}(m)}{m^{1+it}}\Big|^{2}dt\Big) \\&+(\log X)^{-A}. \end{split} \nonumber \end{equation}

\end{Lemma}
\begin{proof}
The proof of this basically follows from \cite[Lemma 14]{MR3}.  The 1-bounded condition in \cite[Lemma 14]{MR3} can be readily eliminated (as shown in \cite[Lemma 2.1]{doi:10.1142/S1793042122500385}).
\end{proof}
\noindent Let $K$ be the set of integers $m \in S$ such that $p^{2} \nmid m$ for all $p\in [P,Q]\cap \mathbf{P},$ and let $K^{c}=\{ m\in S : m \notin K \}.$
Then 
\begin{equation}\label{325}\sum_{m=X}^{2X} \frac{ \mathcal{A}_{F}(m,1,...,1)1_{S}(m)}{m^{1+it}} = \sum_{m=X, \atop m \in K}^{2X}  \frac{ \mathcal{A}_{F}(m,1,...,1)1_{S}(m)}{m^{1+it}}  +  \sum_{m=X, \atop m \in K^{c}}^{2X}  \frac{ \mathcal{A}_{F}(m,1,...,1)1_{S}(m)}{m^{1+it}}.\end{equation}
The first summand on the right-hand side of \eqref{325} can be represented as
\begin{equation}\label{326}\sum_{p=P}^{Q} \frac{\mathcal{A}_{F}(p,1,...,1)}{p^{1+it}} \sum_{m_{1}=\frac{X}{p}, \atop pm_{1} \in K}^{\frac{2X}{p}} \frac{\mathcal{A}_{F}(m_{1},1,...,1)}{m_{1}^{1+it}}\frac{1}{|\{ q \in [P,Q]\cap \mathbf{P} : q |m_{1}\}|+1}\end{equation} (see \cite[section 17.3]{IK1}, or \cite[(9)]{MR3}).
\begin{Lemma} Assume $\bf{GRP}.$
Let $X,T$ be sufficiently large. Then
\begin{equation} \int_{0}^{T} \left| \sum_{m=X,\atop m\in K^{c}}^{2X} \frac{\mathcal{A}_{F}(m,1,...,1)1_{S}(m)}{m^{1+it}}\right|^{2} dt \ll_{F} 
\left(\frac{T}{X}+1\right) \frac{(\log X)^{\frac{n^{4}-1}{2}}}{P^{\frac{1}{2}}}. \nonumber \end{equation} \end{Lemma}
\begin{proof}
By applying Lemma \ref{Lemma23} and the Cauchy-Schwarz inequality, 
\begin{equation}\label{327} \begin{split}\int_{0}^{T} \bigg| \sum_{m=X \atop m\in K^{c}}^{2X} \frac{\mathcal{A}_{F}(m,1,...,1)1_{S}(m)}{m^{1+it}}&\bigg|^{2} dt \ll \frac{(T+O(X))}{X^{2}} \sum_{m \in K^{c}} |\mathcal{A}_{F}(m,1,...,1)|^{2}1_{S}(m) \\&\ll  \frac{(T+O(X))}{X^{2}} \left(\sum_{m=X}^{2X} |\mathcal{A}_{F}(m,1,...,1)|^{4} \right)^{\frac{1}{2}} \left(\sum_{m \in K^{c}} 1_{S}(m) \right)^{\frac{1}{2}}. \end{split}\end{equation}
By the definition of $K^{c},$ we see that 
\begin{equation}\label{328}\sum_{m \in K^{c}} 1_{S}(m) \ll \sum_{p=P}^{Q} \sum_{m=\frac{X}{p} ,\atop p|m}^{\frac{2X}{p}}  1 
 \ll X \sum_{p=P}^{Q} \frac{1}{p^{2}} \ll \frac{X}{P}.\end{equation}
Therefore, by combining \eqref{315} and \eqref{328}, the last term in \eqref{327} is bounded by 
$$\ll_{F}(T/X+1) \frac{(\log X)^{\frac{n^{4}-1}{2}}}{P^{\frac{1}{2}}}.$$
\end{proof}

\begin{Lemma}\label{Lemma38} Assume $\bf{GRP}$. Let $X$ be sufficiently large, and let $2 \leq H \ll_{\epsilon} X^{1-\epsilon}.$ For $v \in I:=[ H\log P-1, H\log Q],$ define \begin{equation} \begin{split} &\mathcal{N}_{v,H}(s):= \sum_{p=e^{\frac{v}{H}}}^{e^{\frac{v+1}{H}}} \frac{\mathcal{A}_{F}(p,1,...,1)}{p^{s}} \\ &\mathcal{R}_{v,H}(s):= \sum_{m_{1}=Xe^{-\frac{v}{H}}}^{2Xe^{-\frac{v}{H}} } \frac{\mathcal{A}_{F}(m_{1},1,...,1)}{m_{1}^{s}} \frac{1}{|\{q \in [P,Q] \cap \mathbf{P} : q|m_{1}\}|+1}. \end{split} \end{equation} Then
\begin{equation}\begin{split} \int_{e^{(\log X)^{2\epsilon}}}^{T}&\bigg| \sum_{m=X}^{2X}  \frac{\mathcal{A}_{F}(m,1,...,1)1_{S}(m)}{m^{1+it}}\bigg|^{2} dt \ll_{F,\epsilon}  \frac{T+X}{HX}(\log X)^{n^{2}-1} \\&+\frac{T+X}{XP^{\frac{1}{2}}} (\log X)^{n^{4}-1} + H \log \frac{Q}{P} \sum_{j \in I} \int_{e^{(\log X)^{2\epsilon}}}^{T} |\mathcal{N}_{j,H}(1+it)|^{2}|\mathcal{R}_{j,H}(1+it)|^{2}dt.  \end{split}\end{equation} 
\end{Lemma}
\begin{proof} The proof of this basically follows from \cite[Lemma 12]{MR} . 
By applying \eqref{325}, \eqref{326}, we see that 
\begin{equation} \begin{split}\sum_{m=X}^{2X} \frac{\mathcal{A}_{F}(m,1,...,1)1_{S}(m)}{m^{1+it}}&=
\sum_{j \in I} \mathcal{N}_{j,H}(1+it)\mathcal{R}_{j.H}(1+it)  + \sum_{2X \leq m \leq 2Xe^{\frac{1}{H}}} \frac{b(m)}{m^{1+it}} \\ &+\sum_{m=X,\atop m\in K^{c}}^{2X} \frac{\mathcal{A}_{F}(m,1,...,1)1_{S}(m)}{m^{1+it}} + \sum_{p=P}^{Q} \sum_{m=X, \atop p^{2}|m}^{2X} \frac{c(m)}{m^{1+it}}\end{split} \end{equation}
where $b(m) \ll |\mathcal{A}_{F}(m,1,...,1)|$ and $c(m)\ll |\mathcal{A}_{F}(p,1,..,1) \mathcal{A}_{F}(m_{1},1,..,1)|$ for some $p\in[P,Q], pm_{1}=m.$ $b(n)$ and $c(n)$ arise from  the difference  
$$\sum_{m=X,\atop m\in K}^{2X} \frac{\mathcal{A}_{F}(m,1,...1)1_{S}(m)}{m^{1+it}} - \sum_{j \in I} \mathcal{N}_{j,H}(1+it)\mathcal{R}_{j.H}(1+it).$$
By using Lemma \ref{Shiutheorem}, we have
\begin{equation}\begin{split}\sum_{m=2X}^{2Xe^{\frac{1}{H}}} \left|\frac{\mathcal{A}_{F}(m,1,...,1)}{m}\right|^{2} &\ll_{F,\epsilon}  \frac{e^{\frac{2}{H}}}{X^{2}} \sum_{m=2X}^{2Xe^{\frac{1}{H}}} |\mathcal{A}_{F}(m,1,...,1)|^{2} \\ &\ll_{F,\epsilon} \frac{e^{\frac{2}{H}}}{X^{2}} \frac{X}{H} (\log X)^{n^{2}-1}. \nonumber \end{split}\end{equation}
Therefore, by using Lemma \ref{Lemma23}, we see that 
\begin{equation}\begin{split} \int_{e^{(\log X)^{2\epsilon}}}^{T} \left|  \sum_{m=2X}^{2Xe^{\frac{1}{H}}} \frac{b(m)}{m^{1+it}} \right|^{2}dt &\ll (T+X)  \sum_{m=2X}^{2Xe^{\frac{1}{H}}} \left|\frac{b(m)}{m}\right|^{2} \\ &\ll (T+X)  \sum_{m=2X}^{2Xe^{\frac{1}{H}}} \left|\frac{\mathcal{A}_{F}(m,1,...,1)}{m}\right|^{2} \\ &\ll_{F,\epsilon} \frac{T+X}{HX} \left(\log X\right)^{n^{2}-1}. \end{split} \end{equation}
By using \eqref{315} and Lemma \ref{Lemma23}, we have 
\begin{equation} \begin{split} \int_{e^{(\log X)^{2\epsilon}}}^{T} \left|\sum_{p=P}^{Q} \sum_{m=X, \atop p^{2}|m}^{2X} \frac{c(m)}{m^{1+it}} \right|^{2} dt &\ll \frac{T+X}{X^{2}}  \left( \sum_{m=1, \atop p^{2}| m \textrm{\thinspace \thinspace for \thinspace some \thinspace} p \in [P,Q]}^{2X} \big|c(m)\big|^{2}\right) \\&\ll _{F}
\frac{T+X}{X^{2}} \left(\frac{X}{P}\right)^{\frac{1}{2}}  \left(\sum_{m=1}^{2X}  \big|\mathcal{A}_{F}(m,1,...,1)\big|^{4}\right)^{\frac{1}{2}} \\ & \ll_{F,\epsilon} \frac{T+X}{X^{2}} \frac{X (\log X)^{\frac{n^{4}-1}{2}}}{P^{\frac{1}{2}}}. \end{split}
\end{equation}
Finally, by applying the Cauchy-Schwarz inequality, we see that
\begin{equation} \begin{split}\int_{e^{(\log X)^{2\epsilon}}}^{T} & \Big|\sum_{j \in I} \mathcal{N}_{j,H}(1+it)\mathcal{R}_{j.H}(1+it)  \Big|^{2} dt \\&\ll (H\log Q- H\log P) \sum_{j \in I} \int_{e^{(\log X)^{2\epsilon}}}^{T} |\mathcal{N}_{j,H}(1+it)|^{2}| \mathcal{R}_{j,H}(1+it)|^{2} dt.
\end{split}\end{equation}
\end{proof}
\noindent
Let us consider 
\begin{equation}\label{335} \sum_{j \in I} \int_{e^{(\log X)^{2\epsilon}}}^{T} |\mathcal{N}_{j,H}(1+it)|^{2}|\mathcal{R}_{j,H}(1+it)|^{2} dt. \end{equation}
For $e^{(\log X)^{2\epsilon}}\ll T \ll X,$ by using Lemma \ref{Lemma12}, we have $|N_{j,H}(1+it)| \ll_{F,A} (\log X)^{-A}.$ Therefore,  
\eqref{335} is bounded by $$ (\log X)^{-2A} \sum_{j \in I} \int_{e^{(\log X)^{2\epsilon}}}^{T} |\mathcal{R}_{j,H}(1+it)|^{2}dt$$
for $e^{(\log X)^{2\epsilon}}\ll T \ll X.$
Using Lemma \ref{Lemma22}, Lemma \ref{Lemma23} and \eqref{317}, we see that 
  $$\int_{e^{(\log X)^{2\epsilon}}}^{T} |\mathcal{R}_{j,H}(1+it)|^{2}dt \ll_{F} (T+Xe^{-\frac{v}{H}})/(Xe^{-\frac{v}{H}}).$$
Therefore, \eqref{335} is bounded by 
\begin{equation}\label{336} H\log \frac{Q}{P}(\log X)^{-2A}(\frac{TQ}{X}+1). \end{equation}
\subsection{Proof of Theorem \ref{Theo1}}
\begin{proof}
Let us split the integrand in \eqref{316} by 
\begin{equation} \begin{split}
&\left(\frac{1}{h_{1}} \sum_{m=x}^{x+h_{1}} \mathcal{A}_{F}(m,1,...,1)1_{S}(m) - \frac{1}{h_{2}} \sum_{m=x}^{x+h_{2}} \mathcal{A}_{F}(m,1,...,1)1_{S}(m)\right) \\&+
\left(\frac{1}{h_{1}} \sum_{m=x}^{x+h_{1}} \mathcal{A}_{F}(m,1,...,1)1_{S'}(m) - \frac{1}{h_{2}} \sum_{m=x}^{x+h_{2}} \mathcal{A}_{F}(m,1,...,1)1_{S'}(m)\right).
\end{split}\nonumber \end{equation}
By using Lemma \ref{Lemma 3.6}, we see that  
\begin{equation}\frac{1}{X} \int_{X}^{2X} \left|\frac{1}{h_{1}}\sum_{m=x}^{x+h_{1}} \mathcal{A}_{F}(m,1,...,1)1_{S}(m)- \frac{1}{h_{2}}\sum_{m=x}^{x+h_{2}} \mathcal{A}_{F}(m,1,...,1)1_{S}(m)   \right|^{2}dx \end{equation}
is bounded by

\begin{equation}\label{338} \begin{split} \ll&\max_{i=1,2} \bigg( \int_{e^{(\log X)^{2\epsilon}}}^{Xh_{i}^{-1}} \left|\sum_{m=X}^{2X} \frac{\mathcal{A}_{F}(m,1,...,1)1_{S}(m)}{m^{1+it}}\right|^{2}dt \\&+ \max_{T>Xh_{i}^{-1}} \frac{Xh_{i}^{-1}}{T} \int_{T}^{2T} \left|\sum_{m=X}^{2X} \frac{\mathcal{A}_{F}(m,1,...,1) 1_{S}(m)}{m^{1+it}}\right|^{2}dt\bigg) \\&+(\log X)^{-A}. \end{split} \end{equation}
By applying Lemma \ref{Lemma38}, \eqref{336}, and choosing $H= (\log X)^{\frac{A}{2}},$ we have 
 \begin{equation}\label{339} \begin{split} \int_{e^{(\log X)^{2\epsilon}}}^{Xh_{i}^{-1}} \left|\sum_{m=X, \atop m\in K}^{2X} \frac{\mathcal{A}_{F}(m,1,...,1)1_{S}(m)}{m^{1+it}}\right|^{2}dt &\ll_{F,\epsilon} (\log X)^{1/3-\epsilon-\frac{3A}{2}} \left(\frac{Q}{h_{i}}+1\right)+ \frac{(\log X)^{\frac{n^{4}-1}{2}}}{P^{\frac{1}{2}}} \\&+\frac{(\log X)^{n^{2}-1}}{H} \\ &\ll_{F,\epsilon} (\log X)^{-\frac{A}{2}+ 3} \end{split}  \end{equation} for $i=1,2,$ and for sufficiently large $A.$ 
Using a similar argument as in \eqref{339}, we see that
$$\max_{Xh_{i}^{-1}<T \leq X} \frac{Xh_{i}^{-1}}{T} \int_{T}^{2T} \left|\sum_{m=X, \atop m\in K}^{2X} \frac{\mathcal{A}_{F}(m,1,...,1) 1_{S}(m)}{m^{1+it}}\right|^{2}dt \ll_{F,\epsilon} (\log X)^{-\frac{A}{2}+3} $$
for $i=1,2.$ In addition, by applying \eqref{Rankinselbergappl}, we have 
\begin{equation} \begin{split} \max_{T \geq X } \frac{Xh_{i}^{-1}}{T} \int_{T}^{2T} \left|\sum_{m=X, \atop m\in K}^{2X} \frac{\mathcal{A}_{F}(m,1,...,1) 1_{S}(m)}{m^{1+it}}\right|^{2}dt &\ll_{F,\epsilon}  \max_{T \geq X } \frac{h_{i}^{-1}}{T} \frac{T+X}{X} \sum_{m=X}^{2X}  |\mathcal{A}_{F}(m,1,...,1) |^{2} \\&\ll_{F,\epsilon} h_{i}^{-1}\end{split} \nonumber \end{equation} for $i=1,2.$
Therefore, \eqref{338} is bounded by \begin{equation}\label{340} (\log X)^{-\frac{A}{3}}.\end{equation}
By using Lemma \ref{Lemma34}, we see that 
\begin{equation}\label{341}\begin{split}\frac{1}{X} &\int_{X}^{2X} \left|\frac{1}{h_{1}}\sum_{m=x}^{x+h_{1}} \mathcal{A}_{F}(m,1,...,1)1_{S'}(m)- \frac{1}{h_{2}}\sum_{m=x}^{x+h_{2}} \mathcal{A}_{F}(m,1,...,1)1_{S'}(m)   \right|^{2}dx  \\&\ll_{F,\epsilon} (\log X)^{-2/3+\epsilon}.\end{split}\end{equation}
Combining \eqref{340} and \eqref{341}, we get \eqref{316}.
\end{proof}
\subsection{Proof of Theorem \ref{Theorem 3.10}}
\begin{proof}
By using the Chebyshev inequality and Theorem \ref{Theo1}, we see that 
$$\frac{1}{h}\sum_{m=x}^{x+h} \mathcal{A}_{F}(m,1,...,1)- \frac{1}{X^{1-\epsilon}}\sum_{m=x}^{x+X^{1-\epsilon}} \mathcal{A}_{F}(m,1,...,1) \ll_{F,\epsilon} (\log X)^{-1/3+\epsilon}$$ 
for almost all $x \in [X,2X].$
By \eqref{342}, we have \begin{equation} \frac{1}{X^{1-\epsilon}}\sum_{m=x}^{x+X^{1-\epsilon}} \mathcal{A}_{F}(m,1,...,1) \ll_{F,\epsilon} X^{\frac{n^{2}-n}{n^{2}+1}-1+\epsilon} \ll_{F,\epsilon}  (\log X)^{-1/3+\epsilon}.   \nonumber \end{equation} 
\end{proof}
\subsection{Proof of Corollary \ref{Cor2}}
\noindent In this subsection, we prove Corollary \ref{Cor2} by applying Theorem \ref{Theo1}.
\begin{proof} For convenience we assume that $X \in \mathbb{N}.$
Let us consider \begin{equation}\label{344} \frac{1}{X}\sum_{x=X}^{2X} \Big|\frac{1}{h} \sum_{m=x}^{x+h} \mathcal{A}_{F}(m,1,...,1)\Big|^{2}.\end{equation}
By squaring out the inner sums and applying the upper bound $\mathcal{A}_{F}(m,1,...,1) \ll_{\epsilon} m^{\epsilon},$ 
\eqref{344} can be represented as 
\begin{equation}  \frac{1}{Xh^{2}}\sum_{l=-h \atop l \neq 0}^{h} |h-l|  \left(\sum_{m=X}^{2X}\mathcal{A}_{F}(m,1,...,1) \overline{\mathcal{A}_{F}(m+l,1,...,1)} \right) + O_{F}\left(\frac{1}{h}\right)+O_{F,\epsilon}\left(hX^{\frac{\epsilon}{2}-1}\right). \nonumber  \end{equation}
By applying Theorem \ref{Theo1} and \eqref{342}, we see that \eqref{344} is bounded by 
$$\ll_{F,\epsilon} (\log X)^{-2/3+\epsilon}.$$
Since $\frac{1}{h}, hX^{\frac{\epsilon}{2}-1} \ll_{\epsilon} (\log X)^{-2/3+\epsilon},$ the proof is completed.
\end{proof}
\subsection{Proof of Corollary \ref{Cor16}}
\begin{proof}
Let $h'=e^{(\log X)^{1-\epsilon}}.$
By using \eqref{1122}, we have
\begin{equation}
\left|\{ x \in [\frac{X}{2},X-h') : \sum_{m=x}^{x+h'} |\mathcal{A}_{F}(m,1,...,1)|> \frac{h' (\log X)^{\frac{1}{n}-1}}{\log \log X}\}\right| \gg X.
\end{equation}
By combining the above inequality with Remark \ref{Remark15}, we see that 
\begin{equation}\label{Cancellation} \sum_{m=x}^{x+h'} \mathcal{A}_{F}(m,1,...,1) \ll_{F,\epsilon} h' (\log X)^{\theta+\epsilon-1} =o_{F,\epsilon} \left(\frac{h' (\log X)^{\frac{1}{n}-1}}{(\log \log X)^{2}}\right) =o_{F,\epsilon} \left(\sum_{m=x}^{x+h'} |\mathcal{A}_{F}(m,1,...,1)|\right)
\end{equation} 
for a positive proportion of $x \in [\frac{X}{2},X-e^{(\log X)^{1-\epsilon}}).$ 
Therefore, there are $x\leq m_{1} <m_{2} \leq x+h'$ such that 
$\mathcal{A}_{F}(m_{1},1,...,1)\mathcal{A}_{F}(m_{2},1,...,1)<0$ 
for a positive proportion of $x \in [\frac{X}{2},X-e^{(\log X)^{1-\epsilon}}).$ 
Hence, the number of sign changes is $\gg_{F,\epsilon} \frac{X}{h'^{2}}.$
\end{proof}
\section{Acknowledgements}
\noindent The author would like to thank his advisor Professor Xiaoqing Li for her helpful advice on zero-free regions and the generalized Ramanujan-Petersson conjecture. The author wishes to express his thanks to Professor Andrew Granville, Professor Kaisa Matomaki, and the anonymous referee for their helpful comments.

\bibliographystyle{plain}   
\bibliography{personal1}  
\end{document}